\documentclass[12pt,twoside]{amsart}

\usepackage{amsfonts}
\usepackage{xr}
\usepackage{graphicx}
\usepackage{amsmath}
\usepackage{epsfig}
\usepackage{color}
\def\R{\mathbb{R}}
\def\N{\mathbb{N}}

\def\epsilon{\varepsilon}

\def\tilde{\widetilde}

\def\trait (#1) (#2) (#3){\vrule width #1pt height #2pt depth #3pt}
\def\fin{\hfill\trait (0.1) (5) (0) \trait (5) (0.1) (0) \kern-5pt 
\trait (5) (5) (-4.9) \trait (0.1) (5) (0)}

\newcommand{\be}{\begin{equation}}
\newcommand{\ee}{\end{equation}}
\newcommand{\baa}{\begin{array}}
\newcommand{\eaa}{\end{array}}
\newcommand{\ba}{\begin{eqnarray}}
\newcommand{\ea}{\end{eqnarray}}

\newtheorem{theo}{\bf Theorem}[section]
\newtheorem{lem}[theo]{\bf Lemma}
\newtheorem{pro}[theo]{\bf Proposition}
\newtheorem{cor}[theo]{\bf Corollary}

\newtheorem{rem}[theo]{\bf Remark}

\title[Nonlocal Poincar\'e inequalities]{ Non local Poincar\'e inequalities on Lie groups with polynomial volume growth }
\author{Emmanuel Russ}
\author{Yannick Sire}

\begin{document}

\maketitle

\begin{abstract} Let $G$ be a real connected Lie group with polynomial volume growth, endowed with its Haar measure $dx$.  Given a $C^2$ positive function $M$ on $G$, we give a sufficient condition for an $L^2$ Poincar\'e inequality with respect to the measure $M(x)dx$ to hold on $G$. We then establish a non-local Poincar\'e inequality on $G$ with respect to $M(x)dx$.
\end{abstract} 
\tableofcontents

\section{Introduction}

Let $G$ be a  unimodular connected Lie group endowed with a measure $M(x)\, dx$ where $M \in L^1(G)$ and $dx$ stands for the Haar measure on $G$. By ``unimodular'', we mean that the Haar measure is left and right-invariant. We always assume that $M=e^{-v}$ where $v$ is a $C^2$ function on $G$. If we denote by $\mathcal  G$  the Lie algebra  of $G$, we consider a family 
$$\mathbb X= \left \{ X_1,...,X_k \right \}$$
of left-invariant vector fields on $G$ satisfying the H\"ormander condition, i.e. $\mathcal G$ is the Lie algebra generated by the $X_i's$. A standard metric on $G$ , called the Carnot-Caratheodory metric, is naturally associated with $\mathbb X$ and is defined as follows: let $\ell : [0,1] \to G$ be an absolutely continuous path. We say that $\ell$ is admissible if there exist measurable functions $a_1,...,a_k : [0,1] \to \mathbb C$ such that, for almost every $t \in [0,1]$, one has 
$$\ell'(t)=\sum_{i=1}^k a_i(t) X_i(\ell(t)).$$
If $\ell$  is admissible, its length is defined by 
$$|\ell |= \int_0^1\left(\sum_{i=1}^k |a_i(t)|^2 \,dt \right)^{ \frac 12 }.$$

For all $x,y \in G $, define $d(x,y)$ as  the infimum of the lengths  of all admissible paths joining $x$ to $y$ (such a curve exists by the H\"ormander condition). This distance is left-invariant. For short, we denote by $|x|$ the distance between $e$, the neutral element of the group and $x$,  so that the distance from $x$ to $y$ is equal to  $|y^{-1}x|$. 

For all $r>0$, denote by $B(x,r)$ the open ball in $G$ with respect to the Carnot-Caratheodory distance and by $V(r)$ the Haar measure of any ball. There exists $d\in \N^{\ast}$ (called the local dimension of $(G,\mathbb X)$) and $0<c<C$ such that, for all $r\in (0,1)$,
$$
cr^d\leq V(r)\leq Cr^d,
$$
see \cite{nsw}. When $r>1$, two situations may occur (see \cite{guivarch}): 
\begin{itemize}
\item Either there exist $c,C,D >0$ such that, for all $r>1$, 
$$c r^D \leq V(r) \leq C r^D$$
where $D$ is called the dimension at infinity of the group (note that, contrary to $d$, $D$ does not depend on $\mathbb X$). The group is said to have polynomial volume growth. 
\item Or  there exist $c_1,c_2,C_1,C_2 >0$ such that, for all $r>1$, 
$$c_1 e^{c_2r} \leq V(r) \leq C_1 e^{C_2r}$$
and the group is said to have exponential volume growth. 
\end{itemize} 
When $G$ has polynomial volume growth, it is plain to see that there exists $C>0$ such that, for all $r>0$,
\begin{equation} \label{homog}
V(2r)\leq CV(r),
\end{equation}
which implies that there exist $C>0$ and $\kappa>0$ such that, for all $r>0$ and all $\theta>1$,
\begin{equation} \label{homogiter}
V(\theta r)\leq C\theta^{\kappa}V(r).
\end{equation}
\bigskip

\noindent  Denote by $H^1(G,d\mu_M)$ the Sobolev space of functions $f\in L^2(G,d\mu_M)$ such that $X_if\in L^2(G,d\mu_M)$ for all $1\leq i\leq k$. We are interested in $L^2$ Poincar\'e inequalities for the measure $d\mu_M$. In order to state sufficient conditions  for such an inequality to hold, we introduce the operator
$$L_M f =-M^{-1} \sum_{i=1}^k X_i \Big \{ M X_i f \Big \}$$ 
for all $f$ such that 
$$f\in {\mathcal D}(L_M):=\left\{g\in H^1(G,d\mu_M);\ \frac 1{\sqrt{M}} X_i \Big \{ M X_i f \Big \}\in L^2(G,dx),\,\, \forall 1\,\leq i\leq k\right\}.$$
One therefore has, for all $f\in {\mathcal D}(L_M)$ and
$g\in H^1(G,d\mu_M)$,
$$
\int_{G} L_Mf(x)g(x) d\mu_M(x)=\sum_{i=1}^k \int_{G} X_i  f(x)\cdot X_i g(x) d\mu_M(x).
$$
In particular, the operator $L_M$ is symmetric on $L^2(G,d\mu_M)$.\par
\noindent Following \cite{bbcg}, say that a $C^2$ function $W:G\rightarrow \R$ is a Lyapunov function if $W(x)\geq 1$ for all $x\in G$ and there exist constants $\theta>0$, $b\geq 0$ and $R>0$ such that, for all $x\in G$,
\begin{equation} \label{lyap}
-L_MW(x)\leq -\theta W(x)+b{\bf 1}_{B(e,R)}(x),
\end{equation}
where, for all $A\subset G$, ${\bf 1}_A$ denotes the characteristic function of $A$. We first claim:
\begin{theo} \label{poincmu}
Assume that $G$ is unimodular and  that there exists a Lyapunov function $W$ on $G$. Then, $d\mu_M$ satisfies the following $L^2$ Poincar\'e inequality: there exists $C>0$ such that, for all function $f\in H^1(G,d\mu_M)$ with $\int_G f(x)d\mu_M(x)=0$,
\begin{equation} \label{eqpoincmu}
\int_G \left\vert f(x)\right\vert^2 d\mu_M(x)\leq C\sum_{i=1}^k \int_G \left\vert X_if(x)\right\vert^2 d\mu_M(x).
\end{equation} 
\end{theo}
Let us give, as a corollary, a sufficient condition on $v$ for (\ref{eqpoincmu}) to hold:
\begin{cor} \label{suffpoincmu}
Assume that $G$ is unimodular and there exist constants $a\in (0,1)$, $c>0$ and $R>0$ such that, for all $x\in G$ with $\left\vert x\right\vert>R$,
\begin{equation} \label{assumpoinc}
a\sum_{i=1}^k \left\vert X_iv(x)\right\vert^2-\sum_{i=1}^k X_i^2v(x)\geq c.
\end{equation}
Then (\ref{eqpoincmu}) holds.
\end{cor}
Notice that, if (\ref{assumpoinc}) holds with $a\in \left(0,\frac 12\right)$, then the Poincar\'e inequality (\ref{eqpoincmu}) has the following self-improvement:
\begin{pro} \label{poincimprovedM} 
Assume that $G$ is unimodular and that there exist constants $c>0$, $R>0$ and $\varepsilon\in (0,1)$ such that, for all $x\in G$,
\begin{equation} \label{c}
\frac{1-\varepsilon}2\sum_{i=1}^k \left\vert X_iv(x)\right\vert^2
-\sum_{i=1}^k X_i^2v(x)\geq c\mbox{ whenever }\left\vert x\right\vert>R.
\end{equation}
Then there exists $C>0$ such that, for all function $f\in H^1(G,d\mu_M)$ such that $\int_G f(x)d\mu_M(x)=0$:
\begin{equation} \label{pim} 
 \sum_{i=1}^k \int_{G} \left\vert X_if(x)\right\vert^2d\mu_M(x)
 \ge C \, \int_{G} \left\vert f(x)\right\vert
 ^2\left(1+\sum_{i=1}^k \left\vert X_iv(x)\right\vert^2\right)d\mu_M(x)
\end{equation}
\end{pro}
We finally obtain a Poincar\'e inequality for $d\mu_M$ involving a non local term:
\begin{theo} \label{mainth} 
Let $G$ be  a  unimodular Lie group with polynomial growth.
Let $d\mu_M= M dx$ be a measure absolutely continuous with respect to the Haar measure on $G$ where $M=e^{-v} \in L^1(G)$ and $v\in C^2(G)$. Assume that there exist constants $c>0$, $R>0$ and $\varepsilon\in (0,1)$ such that (\ref{c}) holds.
Let $\alpha\in
(0,2)$. Then there exists $\lambda_\alpha(M)>0$  such that, for any function $f\in {\mathcal D}(G)$
satisfying $\int_{G} f(x) \, d\mu_M(x)=0$,
 \begin{eqnarray} \label{poincfrac} \iint _{G \times G }
   \frac{\left\vert f(x)-f(y)\right\vert^2}{V\left(\left\vert
       y^{-1} x\right\vert\right)\left\vert
       y^{-1} x\right\vert^{\alpha}} \, dx\, d\mu_M(y) \ge
   \lambda_\alpha(M) \\
   \, \int_{\R^n} \left\vert
     f(x)\right\vert^2 \left(1+\sum_{i=1}^k \left\vert X_iv(x)\right\vert^2\right) \,d\mu_M(x). \nonumber
\end{eqnarray}
\end{theo}

Note that  (\ref{poincfrac}) is an improvement of (\ref{pim}) in terms of fractional nonlocal quantities. The proof follows the same line  as the paper \cite{MRS} but we concentrate here on a more geometric context. 

In order to prove Theorem \ref{mainth}, we need to introduce fractional powers of $L_M$. This is the object of the following developments. Since the operator $L_M$ is symmetric and non-negative on $L^2(G,d\mu_M)$, we can define the
usual power $L^{\beta}$ for any $\beta\in (0,1)$ by means of spectral
theory.\par
\noindent Section \ref{poinc1} is devoted to the proof of Theorem \ref{poincmu} and Corollary \ref{suffpoincmu}. Then, in Section \ref{proofmain}, we check $L^2$ ``off-diagonal'' estimates for the resolvent of $L_M$ and use them to establish Theorem \ref{mainth}.

\section{A proof of the Poincar\'e inequality for $d\mu_M$} \label{poinc1}
We follow closely the approach of \cite{bbcg}. Recall first that the following $L^2$ local Poincar\'e inequality holds on $G$ for the measure $dx$: for all $R>0$, there exists $C_R>0$ such that, for all $x\in G$, all $r \in (0,R)$, all ball $B:=B(x,r)$  and all function $f\in C^{\infty}(B)$,
\begin{equation} \label{poincdx}
\int_B \left\vert f(x)-f_B\right\vert^2dx\leq C_R r^2\sum_{i=1}^k \int_B \left\vert X_if(x)\right\vert^2dx,
\end{equation}
where $f_B:=\frac 1{V(r)}\int_B f(x)dx$.  In the Euclidean context,  Poincar\'e inequalities for vector-fields satisfying H\"ormander conditions were obtained by Jerison in \cite{jerison}. A proof of (\ref{poincdx}) in the case of unimodular Lie groups can be found in \cite{saloff1995parabolic}, but the idea goes back to \cite{varo}. A nice survey on this topic can be found in \cite{sobmetpoinc}. Notice that no global growth assumption on the volume of balls is required for (\ref{poincdx}) to hold.  \par
\bigskip

\noindent The proof of (\ref{eqpoincmu}) relies on the following inequality:
\begin{lem} \label{lempoinc}
For all function $f\in H^1(G,d\mu_M)$ on $G$,
\begin{equation} \label{firstpart}
\int_G \frac{L_MW}{W}(x)f^2(x)d\mu_M(x) \leq \sum_{i=1}^k \int_G \left\vert X_if(x)\right\vert^2d\mu_M(x).
\end{equation}
\end{lem}
{\bf Proof: } Assume first that $f$ is compactly supported on $G$.  Using the definition of $L_M$, one has
$$
\begin{array}{lll}
\displaystyle \int_G \frac{L_MW}{W}(x)f^2(x)d\mu_M(x) & = & \displaystyle \sum_{i=1}^k \int_G X_i\left(\frac{f^2}W\right)(x) \cdot X_iW(x)d\mu_M(x)\\
& = & \displaystyle  2\sum_{i=1}^k \int_G \frac fW(x) X_if(x)\cdot X_iW(x)d\mu_M(x)\\
& & \displaystyle -\sum_{i=1}^k \int_G \frac{f^2}{W^2}(x) \left\vert X_iW(x)\right\vert^2 d\mu_M(x)\\
& = & \displaystyle \sum_{i=1}^k \int_G \left\vert X_i f(x)\right\vert^2d\mu_M(x) \\
& & \displaystyle -\sum_{i=1}^k \int_G \left\vert X_if-\frac fW X_iW\right\vert^2(x)d\mu_M(x)\\
& \leq & \displaystyle \sum_{i=1}^k \int_G \left\vert X_if(x)\right\vert^2d\mu_M(x).
\end{array}
$$
Notice that all the previous integrals are finite because of the support condition on $f$. Now, if $f$ is as in Lemma \ref{lempoinc}, consider a nondecreasing sequence of smooth compactly supported functions $\chi_n$ satisfying
$$
{\bf 1}_{B(e,nR)} \leq \chi_n\leq 1\mbox{ and } \left\vert X_i\chi_n\right\vert\leq 1\mbox{ for all }1\leq i\leq k.
$$
Applying (\ref{firstpart}) to $f\chi_n$ and letting $n$ go to $+\infty$ yields the desired conclusion, by use of the monotone convergence theorem in the left-hand side and the dominated convergence theorem in the right-hand side. \hfill\fin\par

\medskip

\noindent Let us now establish (\ref{eqpoincmu}). Let $g$ be a smooth function on $G$ and let $f:=g-c$ on $G$ where $c$ is a constant to be chosen. By assumption (\ref{lyap}), 
\begin{equation} \label{twoterms}
\int_G f^2(x)d\mu_M(x)\leq \int_G f^2(x)\frac{L_MW}{\theta W}(x)d\mu_M(x)+\int_{B(e,R)} f^2(x)\frac{b}{\theta W}(x)d\mu_M(x).
\end{equation}
\par
\noindent Lemma \ref{lempoinc} shows that (\ref{firstpart}) holds.  Let us now turn to the second term in the right-hand side of (\ref{twoterms}). Fix $c$ such that $\int_{B(e,R)} f(x)d\mu_M(x)=0$. By (\ref{poincdx}) applied to $f$ on $B(e,R)$ and the fact that $M$ is bounded from above and below on $B(e,R)$, one has
$$
\int_{B(e,R)} f^2(x)d\mu_M(x)\leq CR^2\sum_{i=1}^k \int_{B(e,R)} \left\vert X_if(x)\right\vert^2d\mu_M(x)
$$
where the constant $C$ depends on $R$ and $M$. 
Therefore, using the fact that $W\geq 1$ on $G$,
\begin{equation} \label{secondpart}
\int_{B(e,R)} f^2(x)\frac{b}{\theta W}(x)d\mu_M(x)\leq CR^2\sum_{i=1}^k \int_{B(e,R)} \left\vert X_if(x)\right\vert^2d\mu_M(x)
\end{equation}
where the constant $C$ depends on $R, M, \theta$ and $b$.  
Gathering (\ref{twoterms}), (\ref{firstpart}) and (\ref{secondpart}) yields
$$
\int_G (g(x)-c)^2d\mu_M(x)\leq C\sum_{i=1}^k \int_G \left\vert X_ig(x)\right\vert^2d\mu_M(x),
$$
which easily implies (\ref{eqpoincmu}) for the function $g$ (and the same dependence for the constant $C$). \hfill\fin\par

\bigskip

\noindent {\bf Proof of Corollary \ref{suffpoincmu}: } according to Theorem \ref{poincmu}, it is enough to find a Lyapunov function $W$. Define
$$
W(x):=e^{\gamma\left(v(x)-\inf_Gv\right)}
$$
where $\gamma>0$ will be chosen later. Since
$$
-L_MW(x)=\gamma\left(\sum_{i=1}^k X_i^2v(x)-(1-\gamma)\sum_{i=1}^k \left\vert X_iv(x)\right\vert^2\right)W(x),
$$
$W$ is a Lyapunov function for $\gamma:=1-a$ because of the assumption on $v$. Indeed, one can take $\theta=c \gamma$ and $b= \max_{B(e,R)} \Big \{-L_M W+\theta W \Big \}$ (recall that $M$ is a $C^2$ function). \hfill\fin\par

\bigskip

\noindent Let us now prove Proposition \ref{poincimprovedM}. Observe first that, since $v$ is $C^2$ on $G$ and (\ref{c}) holds, there exists $\alpha\in \R$ such that, for all $x\in G$,
\begin{equation} \label{alpha}
\frac{1-\varepsilon}2\sum_{i=1}^k \left\vert X_iv(x)\right\vert^2
-\sum_{i=1}^k X_i^2v(x)\geq \alpha.
\end{equation}
Let $f$ be as in the
statement of Proposition \ref{poincimprovedM} and let $g:=fM^{\frac 12}$.
Since, for all $1\leq i\leq k$,
$$
X_if=M^{-\frac 12}X_ig-\frac 12 g M^{-\frac 32} X_iM. 
$$
Assumption (\ref{alpha}) yields two positive constants $\beta,\gamma$ such that
\begin{multline} \label{estim1}
\displaystyle \sum_{i=1}^k \int_{G} \left\vert X_if(x)\right\vert^2(x) \,
d\mu_M(x)= \\
\displaystyle \sum_{i=1}^k \int_{G} \left(\left\vert X_ig(x)\right\vert^2 
+\frac 14g^2(x)\left\vert X_i v(x)\right\vert^2+
g(x)X_ig(x)X_i v(x)\right) \, dx\\
= \displaystyle \sum_{i=1}^k \int_{G} \left(\left\vert X_i
   g(x)\right\vert^2 
+\frac 14g^2(x)\left\vert X_i v(x)\right\vert^2+\frac 12
X_i\left(g^2\right)(x)X_iv(x)\right) \, dx\\
\geq \displaystyle \sum_{i=1}^k \int_{G} g^2(x) \left(\frac 14\left\vert
   X_iv(x)\right\vert^2-\frac 12 X_i^2v(x)\right) \, dx\\
\geq  \displaystyle \sum_{i=1}^k \int_{G} f^2(x)\left(\beta \left\vert
   X_iv(x)\right\vert^2-\gamma\right)d\mu_M(x).
\end{multline}
The conjunction of (\ref{eqpoincmu}), which holds because of (\ref{c}), and (\ref{estim1}) yields the
desired conclusion. \hfill\fin\par

\section{Proof of Theorem \ref{mainth}} \label{proofmain}

We divide the proof into several steps. 

\subsection{Rewriting the improved Poincar\'e inequality}

By the definition of $L_M$, the conclusion of Proposition \ref{poincimprovedM} means, in terms of operators  in $L^2(G,d\mu_M)$, that, for some $\lambda>0$,
\begin{equation} \label{ineqop}
L_M\geq \lambda \mu, 
\end{equation}
where $\mu$ is the multiplication operator by $1+\sum_{i=1}^k \left\vert X_iv\right\vert^2$. Using a functional calculus argument (see \cite{davies}, p.
110), one deduces from (\ref{ineqop}) that, for any $\alpha\in (0,2)$,
\[
L_M^{\alpha/2}\geq \lambda^{\alpha/2}\mu^{\alpha/2}
\]
which implies, thanks to the fact $L_M^{\alpha/2} = (L_M^{\alpha/4})^2$ and the
symmetry of $L_M^{\alpha/4}$ on $L^2(G,d\mu_M)$, that
$$
\int_{G} \left\vert f(x)\right\vert^2 \left(1+\sum_{i=1}^k \left\vert X_iv(x)\right\vert^2\right)^{\alpha/2}d\mu_M(x)\leq  $$
$$C \int_{G} \left\vert L_M^{\alpha/4}f(x)\right\vert^2 \, d\mu_M(x) =
C \left\| L_M^{\alpha/4} f \right\|^2 _{L^2(G,d\mu_M)}.
$$
The conclusion of Theorem \ref{mainth} will follow by estimating the quantity 
$ \left\| L^{\alpha/4} f \right\|^2 _{L^2(G,d\mu_M)}.$

\subsection{Off-diagonal $L^2$ estimates for the resolvent of $L_M$}
The crucial estimates to derive the desired inequality are some $L^2$ ``off-diagonal'' estimates for the resolvent of $L_M$, in the spirit of \cite{gaff} . This is the object of the following lemma.  
\begin{lem} \label{off} There exists $C$ with the following property: for all closed disjoint subsets
$E,F\subset G$ with $\mbox{d}(E,F)=:d>0$, all function $f\in
L^2(G,d\mu_M)$ supported in $E$ and all $t>0$,
$$
\left\Vert (\mbox{I}+t \, L_M)^{-1}f\right\Vert_{L^2(F,d\mu_M)}+\left\Vert t \,
 L_M(\mbox{I}+t \, L_M)^{-1}f\right\Vert_{L^2(F,d\mu_M)}\leq $$
 $$ 8 \, e^{-C \,
 \frac{d}{\sqrt t}} \left\Vert f\right\Vert_{L^2(E,d\mu_M)}.
$$
\end{lem}

\begin{proof} We argue as in \cite{kato}, Lemma
1.1. From the fact that $L_M$ is self-adjoint on $L^2(G,d\mu_M)$ we have
\[ 
\| (L_M-\mu)^{-1} \|_{L^2(G,d\mu_M)} \le
\frac{1}{\mbox{dist}(\mu,\Sigma(L_M))}
\]
where $\Sigma(L_M)$ denotes the spectrum of $L_M$, and $\mu \not \in
\Sigma(L_M)$.  Then we deduce that $(\mbox{I}+t \, L_M)^{-1}$ is bounded with
norm less than $1$ for all $t >0$, and it is clearly enough to argue when
$0<t<d$.

In the following computations, we will make explicit the dependence of the measure $d\mu_M$ in terms of $M$ for sake of clarity.  Define $u_t=(\mbox{I}+t \, L_M)^{-1}f$, so that, for all function $v\in
H^{1}(G,d\mu_M)$,
\begin{eqnarray} \label{test} \int_{G} u_t(x) \, v(x) \, M(x) \,
dx+ \\ \nonumber
t \, \sum_{i=1}^k \int_{G} X_i u_t(x)\cdot X_i v(x) \, M(x) \,
dx=\\  \nonumber
\int_{G} f(x) \, v(x) \, M(x) \, dx.
\end{eqnarray}
Fix now a nonnegative function $\eta\in {\mathcal D}(G)$ vanishing
on $E$. Since $f$ is supported in $E$, applying (\ref{test}) with
$v=\eta^2 \, u_t$ (remember that $u_t\in H^1(G,d\mu_M)$)  yields
\[
\int_{G} \eta^2(x)\left\vert u_t(x)\right\vert^2
\, M(x) \, dx + t \sum_{i=1}^k\, \int_{G} X_i u_t(x)\cdot X_i
(\eta^2u_t) \, M(x) \, dx=0,
\]
which implies
\begin{multline*}
\int_{G} \eta^2(x)\left\vert u_t(x)\right\vert^2 \,
M(x) \, dx + t \, \int_{G} \eta^2(x)\sum_{i=1}^k \left\vert X_i u_t(x)
\right \vert^2 \, M(x) \, dx \\
= -2 \, t \sum_{i=1}^k \, \int_{G} \eta(x) \, u_t(x) \, X_i \eta(x)
\cdot X_i u_t(x) \, M(x) \, dx \\
\leq
\displaystyle t \, \int_{G} \left\vert u_t(x)\right\vert^2  \sum_{i=1}^k |X_i \eta(x) |^2 \, M(x)\, dx
+\\
t \, \int_{G} \eta^2(x)\sum_{i=1}^k  \left\vert X_i u_t(x)\right\vert^2
\, M(x) \, dx,
\end{multline*}
hence
\begin{equation} \label{hence} 
\int_{G} \eta^2(x)\left\vert
u_t(x)\right\vert^2 \, M(x) \, dx\leq t \, \int_{G} \left\vert
u_t(x)\right\vert^2 \sum_{i=1}^k \left\vert X_i \eta(x)\right\vert^2 \, M(x)
\, dx.
\end{equation}
Let $\zeta$ be a nonnegative smooth function on $G$ such that $\zeta=0$ on $E$, so that $\eta := e^{\alpha \, \zeta}-1 \geq 0$ and $\eta$ vanishes on $E$ for some $\alpha >0$ to be chosen. Choosing this particular $\eta$ in \eqref{hence} with
$\alpha>0$ gives
\[
\int_{G} \left\vert e^{\alpha \, \zeta(x)}-1\right\vert^2\left\vert
u_t(x)\right\vert^2 \, M(x) \, dx \leq \]
\[
\alpha^2 \, t \, \int_{G}
\left\vert u_t(x)\right\vert^2 \sum_{i=1}^k \left\vert X_i \zeta(x)\right\vert^2
\, e^{2 \, \alpha \, \zeta(x)} \, M(x) \, dx.
\]
Taking $\alpha= 1/(2 \, \sqrt{t} \, \max_i \left\Vert
X_i\zeta\right\Vert_{\infty})$, one obtains
\[
\int_{G} \left\vert e^{\alpha \, \zeta(x)}-1\right\vert^2\left\vert
u_t(x)\right\vert^2 \, M(x) \, dx\leq \frac 14 \, \int_{G} \left\vert
u_t(x)\right\vert^2 e^{2\, \alpha \, \zeta(x)} \, M(x) \, dx. 
\]

Using the fact that the norm of $(I+tL_M)^{-1}$ is bounded by $1$
uniformly in $t >0$, this gives
\[
\begin{array}{lll}
\displaystyle \left\Vert e^{\alpha\zeta} \, u_t\right\Vert_{L^2(G,d\mu_M)} 
& \leq & \displaystyle \left\Vert \left(e^{\alpha\zeta}-1\right) \,  
 u_t\right\Vert_{L^2(G,d\mu_M)} + \left\Vert u_t\right\Vert_{L^2(G,d\mu_M)} \\
& \leq & \displaystyle \frac 12 \left\Vert e^{\alpha\zeta}
 \, u_t\right\Vert_{L^2(G,d\mu_M)} + \left\Vert f\right\Vert_{L^2(G,d\mu_M)},
\end{array}
\]
therefore
\[
\begin{array}{lll}
\displaystyle \int_{G} 
\left\vert e^{\alpha \, \zeta(x)}\right\vert^2\left\vert
 u_t(x)\right\vert^2 \, M(x) \, dx 
& \leq & \displaystyle 4 \, \int_{G} \left\vert f(x)\right\vert^2
\, M(x) \,dx.
\end{array}
\]
We choose now $\zeta$ such that $\zeta =0$ on $E$ as before and
additionnally that $\zeta=1$ on $F$. It can  furthermore  be chosen with
$\max_{i=1,...k} \left\Vert X_i\zeta\right\Vert_{\infty} \leq C/d$, which yields the
desired conclusion for the $L^2$ norm of $(I+tL_M)^{-1}f$ with a factor $4$ in the right-hand side.  Since $t \,
L_M(\mbox{I}+t \, L_M)^{-1}f=f-(\mbox{I}+t \, L_M)^{-1}f$, the desired inequality with a factor $8$
readily follows. 
\end{proof}

\subsection{Control of $\left\Vert L_M^{\alpha/4}f\right\Vert_{L^2(G,d\mu_M)}$ and conclusion of the proof of Theorem \ref{mainth}}

This is now  the heart of the proof  to reach the conclusion of Theorem \ref{mainth}. The following first lemma is a standard quadratic estimate on powers of subelliptic operators. It is based on spectral theory.

\begin{lem} \label{quadratic} Let $\alpha\in (0,2)$. There exists $C>0$ such that, for all $f\in {\mathcal D}(L_M)$,
\begin{equation} \label{spectral} \left\Vert
L_M^{\alpha/4}f\right\Vert_{L^2(G, d\mu_M )}^2\leq C_3 \,
\int_0^{+\infty} t^{-1-\alpha/2} \left\Vert t \, L_M \,
(\mbox{I} + t \, L_M)^{-1} f\right\Vert_{L^2(G,d\mu_M)}^2 \, dt.
\end{equation}
\end{lem}

We now come to the desired estimate. 
\begin{lem} \label{controllalpha} Let $\alpha\in (0,2)$ . There exists
$C >0$  such that, for all $f\in {\mathcal
 D}(G)$,
$$
\int_0^{\infty} t^{-1-\alpha/2} \left\Vert t \, L_M \, (\mbox{I} + t \, L_M)^{-1}
f\right\Vert_{L^2(G,d\mu_M)}^2 \, dt  \leq 
$$
$$C \, \iint_{G \times G}
\frac{\left\vert f(x)-f(y)\right\vert^2}{V\left(\left\vert y^{-1}x\right\vert\right)\left\vert
 y^{-1}x\right\vert^{\alpha}} \, M(x)\,    \, dx\, dy.
$$
\end{lem}

\begin{proof} Fix $t \in (0, +\infty)$. Following Lemma~\ref{quadratic}, we give an upper bound of
$$\left\Vert t \, L_M \, (\mbox{I} + t \, L_M)^{-1} f\right\Vert_{L^2(G,d\mu_M)}
^2$$
involving first order differences for $f$.  Using (\ref{homog}), one can pick up a countable family $x_j^t$, $j\in \N$, such that the balls
$B\left(x_j^{t},\sqrt{t}\right)$ are pairwise disjoint and 
\begin{equation} \label{union}
G=\bigcup_{j \in \N} B\left(x_j^{t},2\sqrt{t}\right).
\end{equation}
By Lemma \ref{cardinal} in Appendix A, there exists a constant $\tilde C>0$
such that for all $\theta>1$ and all $x\in G$, there are at most $\tilde
C\, \theta^{2\kappa}$ indexes $j$ such that $|x^{-1}x_j^t |\leq
\theta\sqrt{t}$ where $\kappa$ is given by (\ref{homogiter}).

\medskip

For fixed $j$, one has
$$
t \, L_M \, (\mbox{I} + t\, L_M)^{-1} f= t \, L_M \, (\mbox{I} + t\, L_M)^{-1} \,
g^{j,t}
$$
where, for all $x\in G$, 
$$
g^{j,t}(x):=f(x)-m^{j,t}
$$
and $m^{j,t}$ is defined by
$$
m^{j,t}:=\frac 1{V\left(2\sqrt{t}\right)}\int_{B\left(x_j^{t},2\sqrt{t}\right)}
f(y) dy$$
Note that, here, the mean value of $f$ is computed with respect to the
Haar measure on $G$. Since (\ref{union}) holds, one
clearly has
$$
\begin{array}{lll}
\displaystyle \left\Vert t \, L _M\, (\mbox{I} + t\, L_M)^{-1} f 
\right\Vert_{L^2(G,d\mu_M)}^2 
& \leq & \displaystyle \sum_{j \in \N} 
\left\Vert t \, L_M \, (\mbox{I} + t\, L_M)^{-1} f
\right\Vert_{L^2\left(B(x_j^t,2\sqrt{t}),d\mu_M\right)}^2\\
& = & \displaystyle  \sum_{j \in \N} \left\Vert t\, L_M \, (\mbox{I} + t\,
L_M)^{-1}  g^{j,t}
\right\Vert_{L^2\left(B\left(x_j^t,2\sqrt{t}\right),d\mu_M \right )}^2,
\end{array}
$$
and we are left with the task of estimating
$$
\left\Vert t \, L_M \, (\mbox{I} + t\, L_M)^{-1}
g^{j,t}\right\Vert_{L^2\left(B\left(x_j^{t},2\sqrt{t}\right),d\mu_M \right )} ^2. 
$$

To that purpose, set
$$
C_0^{j,t}=B\left(x_j^t, 4\sqrt{t}\right) \ \mbox{ and
} \ C_k^{j,t}=B\left(x_j^t,2^{k+2}\sqrt{t}\right)\setminus
B\left(x_j^t,2^{k+1}\sqrt{t}\right), \ \forall \, k \ge 1,
$$
and $g^{j,t}_k:=g^{j,t} \, {\bf 1}_{C_k^{j,t}}$, $k \ge 0$, where, for any
subset $A\subset G$, ${\bf 1}_A$ is the usual characteristic function of
$A$. Since $g^{j,t}=\sum_{k\geq 0} g^{j,t}_k$ one has
\begin{eqnarray} 
\left\Vert t \, L_M \, (\mbox{I} + t\, L_M)^{-1}  g^{j,t} 
\right\Vert_{L^2\left(B\left(x_j^{t},2\sqrt{t}\right),d\mu_M\right)}
\leq  
\\
\sum_{k\geq 0}  
\left\Vert t \, L_M \, (\mbox{I} + t\, L_M)^{-1}  g_k^{j,t}
\right\Vert_{L^2\left(B\left(x_j^{t},2\sqrt{t}\right),d\mu_M\right)} \nonumber
\end{eqnarray}
and, using Lemma \ref{off}, one obtains (for some constants $C,c>0$)
\begin{eqnarray} \label{expdecay}
\left\Vert t \, L_M \, (\mbox{I} + t\, L_M)^{-1}  g^{j,t} 
\right\Vert_{L^2\left(B\left(x_j^{t},2\sqrt{t}\right),d\mu_M\right)} 
\leq  
\\ \nonumber
C \, \left( \left\Vert
  g_0^{j,t}\right\Vert_{L^2(C_0^{j,t},d\mu_M)}
+\sum_{k\geq 1} e^{-c \, 2^{k}} 
\left\Vert g_k^{j,t}\right\Vert_{L^2(C_k^{j,t},d\mu_M)} \right).
\end{eqnarray}
By Cauchy-Schwarz's inequality, we deduce (for another constant $C'>0$)
\begin{eqnarray}\label{expdecaybis} 
\left\Vert t \, L_M \, (\mbox{I} + t\, L_M)^{-1}  g^{j,t}
\right\Vert_{L^2\left(B\left(x_j^{t},2\sqrt{t}\right),d\mu_M\right)} ^2  \leq \\ \nonumber
C' \,
\left( \left\Vert g_0^{j,t}\right\Vert_{L^2(C_0^{j,t},d\mu_M)}^2 +\sum_{k\geq
  1} e^{-c \, 2^{k}} \left\Vert g_k^{j,t}\right\Vert_{L^2(C_k^{j,t},d\mu_M)}^2
\right).
\end{eqnarray}

As a consequence,  we have 
\begin{equation} \label{expdecayter}
\begin{array}{lll}
\displaystyle \int_0^{\infty} t^{-1-\alpha/2} 
\left\Vert t \, L_M \, (\mbox{I} + t \, L_M)^{-1}
 f\right\Vert_{L^2(G,d\mu_M)}^2 \, dt \leq \\ 
\displaystyle C' \, \int_0^{\infty}
t^{-1-\alpha/2} \sum_{j\ge 0} 
\left\Vert g_0^{j,t}\right\Vert_{L^2(C_0^{j,t},d\mu_M)}^2 dt+ \\
\displaystyle C' \, \int_0^{\infty}  t^{-1-\alpha/2} 
\sum_{k\geq 1} e^{-c \, 2^{k}} \sum_{j \geq 0}
\left\Vert g_k^{j,t}\right\Vert_{L^2(C_k^{j,t},d\mu_M)}^2 dt.
\end{array}
\end{equation}

We claim that, and we pospone the proof into Appendix B:
\begin{lem} \label{estimg}
There exists $\bar C>0$ such that, for all $t>0$ and all $j \in \N$:
\begin{itemize}
\item[{\bf A.}] For the first term:
$$\displaystyle \left\Vert g_0^{j,t}\right\Vert_{L^2(C_0^{j,t},M)}^2\leq  
\frac{\bar C}{V(\sqrt{t})} \int_{B\left(x_j^t,4\sqrt{t}\right)}
\int_{B\left(x_j^t,4\sqrt{t}\right)} \left\vert f(x)-f(y)\right\vert^2 \,
d\mu_M(x) \, dy. $$
\item[{\bf B.}]
For all $k\geq 1$,
\[ \left\Vert g^{j,t}_k\right\Vert_{L^2(C_k^{j,t},d\mu_M)}^2 
\leq \]
\[
\frac{\bar C}{V(2^k\sqrt{t})} \int_{x\in B(x^t_j,2^{k+2}\sqrt{t})} 
\int_{y\in B(x^t_j,2^{k+2}\sqrt{t})} \left\vert f(x)-f(y)\right\vert^2 \, d\mu_M(x)\, dy.\]
\end{itemize}
\end{lem}
We finish the proof of the theorem. Using Assertion {\bf A} in Lemma \ref{estimg},
summing up on $j \ge 0$ and integrating over $(0,\infty)$, we get 
\begin{multline*}
\displaystyle \int_0^{\infty} t^{-1-\alpha/2} \sum_{j \ge 0} \left \Vert
g_0^{j,t}\right\Vert_{L^2\left(C_0^{j,t},d\mu_M\right)}^2 \, dt = \sum_{j
\ge 0} \int_0^{\infty} t^{-1-\alpha/2} \left\Vert g_0^{j,t}\right
\Vert_{L^2\left(C_0^{j,t},d\mu_M\right)}^2 \, dt \\
\displaystyle \le  \bar C \, \sum_{j \ge 0} \int_0^{\infty}
\frac{t^{-1-\frac{\alpha}2}}{V(\sqrt{t})} \left(\int_{B\left(x_j^t,4\sqrt{t}\right)}
\int_{B\left(x_j^t,4\sqrt{t}\right)} \left\vert
  f(x)-f(y)\right\vert^2 \, d\mu_M(x) \, dy\right) \, dt  \\
\displaystyle \le \bar C\, \sum_{j \ge 0}\iint_{(x,y)\in G\times G}
\left\vert f(x)-f(y)\right\vert^2 M(x)\times \\
\left(\int_{ t\geq
  \max\left\{\frac{\left\vert x^{-1}x_j^t\right\vert^2}{16}\,;\
    \frac{\left\vert y^{-1}x_j^t\right\vert^2}{16}\right\}}
\, \frac{t^{-1-\frac{\alpha}2}}{V(\sqrt{t})}dt\right) \, dx \, dy.
\end{multline*}

The Fubini theorem now shows
$$
\sum_{j \ge 0} \int_{ t\geq \max\left\{\frac{\left\vert
    x^{-1}x_j^t\right\vert^2}{16}\, ; \ \frac{\left\vert
    y^{-1}x_j^t\right\vert^2}{16}\right\}} \, \frac{t^{-1-\frac{\alpha}2}}{V(\sqrt{t})} dt
= $$

$$
\int_0^{\infty} \frac{t^{-1-\frac{\alpha}2}}{V(\sqrt{t})} \, \sum_{j \ge 0} {\bf
1}_{ \left(\max\left\{ \frac{\left\vert x^{-1}x_j^t\right\vert^2}{16}\, ; \
  \frac{\left\vert y^{-1}x_j^t\right\vert^2}{16} \right\},+\infty\right)}
(t) \, dt.
$$
Observe that, by Lemma \ref{cardinal}, there is a constant $N \in \N$ such
that, for all $t>0$, there are at most $N$ indexes $j$ such that
$\left\vert x^{-1}x_j^t\right\vert^2< 16\, t$  and $ \left\vert
y^{-1}x_j^t\right\vert^2<16 \, t$, and for these indexes $j$, one has
$ \left\vert x^{-1}y\right\vert<8\sqrt{t}$. It therefore follows that
$$
\sum_{j \ge 0} {\bf 1}_{\left(\max\left\{ \frac{\left\vert
       x^{-1}x_j^t\right\vert^2}{16} \, ; \ \frac{\left\vert
       y^{-1}x_j^t\right\vert^2}{16} \right\},+\infty\right)}(t)\leq N \, {\bf
1}_{ \left(\left\vert x^{-1}y\right\vert^2/64,+\infty\right)}(t),
$$
so that, by (\ref{homog}),
\begin{multline} \label{intg0}
\displaystyle \int_0^{\infty} t^{-1-\alpha/2} 
\sum_{j} \left\Vert
g_0^{j,t}\right\Vert_{L^2\left(C_0^{j,t},d\mu_M\right)}^2 \, dt \\
\leq \bar C \, N \, \iint_{G\times G} 
\left\vert f(x)-f(y)\right\vert^2M(x) 
\left( \int_{\left\vert x^{-1}y\right\vert^2/64}^{\infty} \,
\frac{t^{-1-\frac{\alpha}2}}{V(\sqrt{t})} \, dt\right) \, dx \, dy \\
\displaystyle  \leq  \bar C \, N \, \iint_{G\times G} 
\frac{\left\vert f(x)-f(y)\right\vert^2}
{V\left(\left\vert x^{-1}y\right\vert\right)\left\vert x^{-1}y\right\vert^{\alpha}} \, d\mu_M(x) \, dy.
\end{multline}
Using now Assertion {\bf B} in Lemma \ref{estimg}, we obtain, for all $j
\ge 0$ and all $k\geq 1$,

$$
\begin{array}{l}
\displaystyle \int_0^{\infty} t^{-1-\alpha/2} 
\, \sum_{j \ge 0} \left\Vert g^{j,t}_k\right\Vert_2^2dt  \\
\displaystyle \leq 
\bar C \,  \, \sum_{j \ge 0} 
\int_0^{\infty} \frac{t^{-1-\frac{\alpha}2}}{V\left(2^k\sqrt{t}\right)} \, 
\left(\iint_{ B(x^t_j,2^{k+2}\sqrt{t}) \times  
   B(x^t_j,2^{k+2}\sqrt{t})} \left\vert f(x)-f(y)\right\vert^2 \, M(x) \, dx
 \, dy\right) \, dt\\
\displaystyle \leq \bar C \, \, 
\sum_{j \ge 0} \iint_{x,y\in G} 
\left\vert f(x)-f(y)\right\vert^2 \, M(x) \times \\ \, 
\displaystyle  \left(\int_0^{\infty} \frac{t^{-1-\frac{\alpha}2}}{V(2^k\sqrt{t})} \, 
 {\bf 1}_{\left( \max\left\{\frac{\left\vert x^{-1}x^t_j\right\vert^2}{4^{k+2}}, 
\frac{\left\vert y^{-1}x^t_j\right\vert^2}{4^{k+2}}\right\},+\infty\right)} 
 (t) \, dt\right) \, dx \, dy. 
\end{array}
$$
But, given $t>0$, $x,y\in G$, by Lemma \ref{cardinal} again, there exist
at most $\tilde C \, 2^{2k\kappa}$ indexes $j$ such that
$$
\left\vert x^{-1}x_j^t\right\vert\leq 2^{k+2}\sqrt{t} \ \mbox{ and } \ 
\left\vert y^{-1}x_j^t\right\vert\leq 2^{k+2}\sqrt{t}, 
$$
and for these indexes $j$, $\left\vert x^{-1}y\right\vert\leq
2^{k+3}\sqrt{t}$. As a consequence,
\begin{equation} \label{intgk}
\begin{array}{lll}
\displaystyle 
\int_0^{\infty} \frac{t^{-1-\frac{\alpha}2}}{V(2^k\sqrt{t})} \, 
\sum_{j \ge 0} {\bf 1}_{ \left(\max\left\{\frac{\left\vert
         x^{-1}x^t_j\right\vert^2}{4^{k+2}}, 
\frac{\left\vert x^{-1}x^t_j\right\vert^2}{4^{k+2}}\right\},+\infty\right)}(t)
\, dt  \leq \\
\displaystyle \tilde C \, 2^{2k\kappa} \, 
\int_{ t\geq \frac{\left\vert x^{-1}y\right\vert^2}{4^{k+3}}}  \,
\frac{t^{-1-\frac{\alpha}2}}{V(2^k\sqrt{t})} \, dt \leq \\
\displaystyle \tilde C' \frac{2^{k(2\kappa+\alpha)}} 
{V\left(\left\vert x^{-1}y\right\vert\right)\left\vert x^{-1}y\right\vert^{\alpha}},
\end{array}
\end{equation}
for some other constant $\tilde C' >0$, and therefore
$$
\displaystyle \int_0^{\infty} \frac{t^{-1-\alpha/2}}{V\left(2^k\sqrt{t}\right)} 
\sum_{j} \left\Vert
g_k ^{j,t}\right\Vert_{L^2\left(C_0^{j,t},d\mu_M\right)}^2 \, dt \leq 
$$

$$
\bar C \, \tilde C' \, 2^{k(2\kappa+\alpha)} \, \iint_{G\times G} 
\frac{\left\vert f(x)-f(y)\right\vert^2}
{V\left(\left\vert x^{-1}y\right\vert\right)\left\vert x^{-1}y\right\vert^{\alpha}} \, M(x) \, dx \, dy.
$$

We can now conclude the proof of Lemma \ref{controllalpha}, using
Lemma \ref{quadratic}, (\ref{expdecay}), (\ref{intg0}) and (\ref{intgk}). 
We have proved, by reconsidering \eqref{expdecayter}: 
\begin{multline}
\displaystyle 
\int_0^{\infty} t^{-1-\alpha/2} \left\Vert t \, L_M \, (\mbox{I} + t \, L_M)^{-1}
f\right\Vert_{L^2(G,d\mu_M)}^2 \, dt
\leq 
\\
\displaystyle C' \, \bar C \, N \, \iint_{G\times G} 
\frac{\left\vert f(x)-f(y)\right\vert^2}
{V\left(\left\vert x^{-1}y\right\vert\right) \left\vert x-y\right\vert^{\alpha}} \, M(x) \, dx \, dy \\
+ \displaystyle \sum_{k \ge 1} 
C' \, \bar C \, \tilde C' \, 2^{k(2\kappa+\alpha)} \, e^{-c \, 2^k} \,  
\iint_{G\times G} 
\frac{\left\vert f(x)-f(y)\right\vert^2}
{V\left(\left\vert x^{-1}y\right\vert\right)\left\vert x^{-1}y\right\vert^{\alpha}} \, M(x) \, dx \, dy
\end{multline}
and we deduce that 
$$
\displaystyle 
\int_0^{\infty} t^{-1-\alpha/2} \left\Vert t \, L_M \, (\mbox{I} + t \, L_M)^{-1}
f\right\Vert_{L^2(G,d\mu_M)}^2 \, dt
\le $$
$$
C \, \iint_{G\times G} \frac{\left\vert f(x)-f(y)\right\vert^2}
{V\left(\left\vert x^{-1}y\right\vert\right)\left\vert x^{-1}y\right\vert^{\alpha}} \, d\mu_M(x) \, dy
$$
for some constant $C$ as claimed in the statement.
\end{proof}

\begin{rem}
In the Euclidean context, Strichartz proved in  (\cite{stri}) that, when $0<\alpha<2$, for all $p\in
(1,+\infty)$,
\begin{equation} \label{compars} \left\Vert
 (-\Delta)^{\alpha/4}f\right\Vert_{L^p(\R^n)} \leq C_{\alpha,p}
\left\Vert S_{\alpha}f\right\Vert_{L^p(\R^n)}
\end{equation}
where
$$
S_{\alpha}f(x)=\left(\int_0^{+\infty} \left(\int_B \left\vert
   f(x+ry)-f(x)\right\vert dy\right)^2
\frac{dr}{r^{1+\alpha}}\right)^{\frac 12},
$$
and also (\cite{stein}) 
\begin{equation} \label{compard} \left\Vert
 (-\Delta)^{\alpha/4}f\right\Vert_{L^p(\R^n)} \leq C_{\alpha,p}
\left\Vert D_{\alpha}f\right\Vert_{L^p(\R^n)}
\end{equation}
where
$$
D_{\alpha}f(x)=\left(\int_{\R^n} \frac{\left\vert
   f(x+y)-f(x)\right\vert^2}{\left\vert
   y\right\vert^{n+\alpha}}dy\right)^{\frac 12}.
$$
In \cite{crt}, these inequalities were extended to the setting of a
unimodular Lie group endowed with a sub-laplacian $\Delta$, relying on
semigroups techniques and Littlewood-Paley-Stein functionals. In
particular, in \cite{crt}, the authors use {\it pointwise} estimates of the kernel
of the semigroup generated by $\Delta$. In the present paper, we deal with the operator $L_M$ for which these pointwise estimates are not available, but it turns out that $L^2$ off-diagonal estimates are enough for our purpose. Note that we do not obtain $L^p$ inequalities here. 
\end{rem}

\section{Appendix A: Technical lemma}

We prove the following lemma. 
\begin{lem} \label{cardinal} 
Let $G$ and the $x_j^t$ be as in the proof of Lemma \ref{controllalpha} . Then there exists a constant $\tilde C>0$ with the
following property: for all $\theta>1$ and all $x\in G$, there are at
most $\tilde C\, \theta^{2\kappa}$ indexes $j$ such that $\left\vert
x^{-1}x_j^t\right\vert\leq \theta\sqrt{t}$.
\end{lem} 

\noindent {\bf Proof of Lemma~\ref{cardinal}.} The argument is very simple
(see \cite{kanai}) and we give it for the sake of completeness. Let $x\in G$
and denote
$$I(x):= \left\{j \in \N \, ;\ \left\vert x^{-1}x_j^t\right\vert\leq \theta\sqrt{t}\right\}.$$

Since, for all $j\in I(x)$
$$B\left(x_j^t,\sqrt{t}\right)\subset B\left(x,\left(1+\theta\right)\sqrt{t}\right),$$ and
$$ B\left(x,\sqrt{t}\right)\subset B\left(x_j^t,(1+\theta)\sqrt{t}\right),$$
one has by (\ref{homogiter}) and the fact that the balls $B\left(x_j^t,\sqrt{t}\right)$ are pairwise disjoint,
$$
\begin{array}{lll}
\displaystyle \left\vert I(x)\right\vert V\left(x,\sqrt{t}\right) & \leq & \displaystyle \sum_{j\in I(x)} V\left(x_j^t,\left(1+\theta\right)\sqrt{t}\right)\\
&  \leq & \displaystyle C(1+\theta)^{\kappa} \sum_{j\in I(x)} V\left(x_j^t,\sqrt{t}\right)\\
& \leq & \displaystyle C(1+\theta)^{\kappa} V\left(x,\left(1+\theta\right)\sqrt{t}\right)\\
& \leq & \displaystyle C(1+\theta)^{2\kappa} V\left(x,\sqrt{t}\right)
\end{array}
$$
and we get the desired conclusion. \hfill\fin\par

\section{Appendix B: Estimates for $g_j^t$}

We prove Lemma \ref{estimg}. For all $x\in G$,
$$
\begin{array}{lll}
\displaystyle g_0^{j,t}(x) & = & \displaystyle f(x) 
-\frac 1{V(2\sqrt{t})} 
\int_{B\left(x_j^t,2\sqrt{t}\right)} f(y) \, dy\\
& = & \displaystyle \frac 1{V(2\sqrt t)} 
\int_{B\left(x_j^t,2\sqrt{t}\right)} (f(x)-f(y)) \, dy.
\end{array}
$$
By Cauchy-Schwarz inequality and (\ref{homog}), it follows that
$$
\left\vert g_0^{j,t}(x)\right\vert^2\leq \frac C{V(\sqrt{t})}   \int_{ B\left(x_j^t,4\sqrt{t}\right)} \left\vert f(x)-f(y)\right\vert^2dy.
$$
Therefore,
$$
\left\Vert g_0^{j,t}\right\Vert_{L^2(C_0^{j,t},M)}^2\leq \frac
C{V(\sqrt{t})} \int_{ B\left(x_j^t,4\sqrt{t}\right)}
\int_{ B\left(x_j^t,4\sqrt{t}\right)} \left\vert
f(x)-f(y)\right\vert^2 \, d\mu_M(x) \, dy,
$$
which shows Assertion {\bf A}. We argue similarly for Assertion {\bf B} and
obtain
$$
\displaystyle \left\Vert g^{j,t}_k\right\Vert_{L^2 (C_k ^{j,t},M)}^2 
\leq \displaystyle \frac C{V(2^k\sqrt{t})} \int_{ x\in B(x^t_j,2^{k+2}\sqrt{t})}
\int_{ y\in B(x^t_j,2^{k+2}\sqrt{t})} \left\vert
f(x)-f(y)\right\vert^2 \, d\mu_M(x) \, dy,
$$
which ends the proof.

\bibliographystyle{alpha}                  
   \bibliography{Biblio-puiss-frac-Lie}      
   \medskip  
{\em Emmanuel Russ}--
Universit\'e Paul C\'ezanne, LATP,\\
Facult\'e des Sciences et Techniques, Case cour A\\
Avenue Escadrille Normandie-Niemen, F-13397 Marseille, Cedex 20, France et  \\
CNRS, LATP, CMI, 39 rue F. Joliot-Curie, F-13453 Marseille Cedex 13, France 

\medskip 

{\em Yannick Sire}--
Universit\'e Paul C\'ezanne, LATP,\\
Facult\'e des Sciences et Techniques, Case cour A\\
Avenue Escadrille Normandie-Niemen, F-13397 Marseille, Cedex 20, France et  \\
CNRS, LATP, CMI, 39 rue F. Joliot-Curie, F-13453 Marseille Cedex 13, France.

\end{document}